\documentclass[a4paper, 12pt,reqno]{amsart}
\usepackage[12pt]{extsizes}
\usepackage{amsmath,amssymb}
\usepackage{amscd}
\usepackage{geometry}
\usepackage{amsthm}
\usepackage{euscript}
\usepackage{latexsym}
\usepackage[matrix,arrow,curve,cmtip]{xy}
\usepackage[cp1251]{inputenc}
\usepackage[russian,english]{babel}
\usepackage{mathrsfs}
\usepackage{graphicx}
\usepackage{keyval}
\usepackage{mathtools}
\usepackage[x11names]{xcolor}
\usepackage[symbol*]{footmisc}
\usepackage{verbatim}
\usepackage{fancyhdr}%коллонтитулы%

\usepackage{amssymb,amscd,enumerate,mathrsfs,silence}

\newtheorem{theorem}{Theorem}[section]
\newtheorem{lemma}{Lemma}[section]
\newtheorem{proposition}{Proposition}[section]

\theoremstyle{definition}

\newtheorem{example}[theorem]{Example}

\newcommand{\m}{\mathfrak}

\begin{document}

\pagestyle{plain}
\title{\bf On Common Divisors of Fox Derivatives with towards to Zero Divisors of Group Rings}

\maketitle

\begin{center}
Viktor Lopatkin\footnote{Laboratory of Modern Algebra and Applications, St. Petersburg State University, 14th Line, 29b, Saint Petersburg, Russia, St. Petersburg Department of Steklov Mathematical Institute. \\ The author is supported by the grant of the Government of the Russian Federation for the state support of scientific research carried out under the supervision of leading scientists, agreement 14.W03.31.0030 dated 15.02.2018. 1 \\ \textit{Email address:} \texttt{wickktor@gmail.com}}
\end{center}

\medskip
\begin{abstract}
Using Composition--Diamond Lemma we construct presentations of groups $G = \langle x_1,\ldots,x_n \, | \, r_1,\ldots, r_m \rangle$ with the following property; for a fixed $1 \le i \le n$, and for all $1 \le j \le m$, Fox derivatives ${\partial r_j / \partial x_i}$ have common divisor. It follows that in some cases the group ring $\mathbb{Z}[G]$ has zero divisors.

\end{abstract}

\section*{Introduction}

Let $\mathsf{F}$ be a field, and $G$ a group with torsion, say $g^n = 1$ with $1 < n < \infty$. Consider the group ring $\mathsf{F}[G]$. We then have
\[
 (1-g)(1+g-g^2 + \cdots + (-1)^{n-1}g) = 0
\]
in $\mathsf{F}[G]$.

Assume now that $G$ is a torsion-free group. \textit{Kaplansky's zero divisor conjecture} states the the group ring $\mathsf{F}[G]$ does not contain nontrivial zero divisors, that is, it is a domain.

In this paper we aim to construct groups with nontrivial zero divisors using methods of homological algebra. The main result is Theorem \ref{themainresult}.

\section{Preliminaries}

Let $G$ be a group is generated by $x_1,\ldots, x_n$ and is defined by relations $r_1,\ldots, r_m$, \textit{i.e.,} the $G$ is presented as follows (= a group presentation)
\[
 G = \langle x_1,\ldots, x_n \, | \, r_1,\ldots, r_m \rangle.
\]

Thus, $G \cong F/N$ where $F$ is the free group with basis $X:=\{x_1,\ldots, x_n\}$ and $N$ is the normal closure in $F$ of the set $\{r_1,\ldots, r_m\}$ of words in $X\cup X^{-1}$.

Further, as is customary, $2$-complexes will be specified by means of groups presentations. The \emph{cellular model} of a presentation of $G$ ()is the $2$-complex $\mathcal{K}_G$ that has a single $0$-cell, one $1$-cell for each generator $x_i$ and one $2$-cell fore each relator $r_j$. An orientation of the cells in the one-skeleton $\mathcal{K}_G^{(1)}$ determines an isomorphism $\pi_1\mathcal{K}_G^{(1)} \cong F$. The $2$-cell corresponding to a relator $r_j$ is attached along a based loop in the one-skeleton that spells the word $r_j$. The inclusion $\mathcal{K}_G^{(1)} \subseteq \mathcal{K}_G$ induces a surjection $F \to \pi_1 \mathcal{K}_G$ with kernel $N$. In particular, $\pi_1\mathcal{K}_G$ is canonically isomorphic to $G$, and so $\pi_2\mathcal{K}_G$ is a left $\mathbb{Z}[G]$-module under the homotopy action of $\pi_1 \mathcal{K}_G$.

Consider the $\mathbb{Z}[G]$-modules $\bigoplus_{i=1}^m\mathbb{Z}[G]$, $\bigoplus_{i=1}^n\mathbb{Z}[G]$. Define
\[
 \mathrm{d}_0: \bigoplus_{i=1}^n\mathbb{Z}[G] \to \mathbb{Z}[G]
\]
by setting
\[
 \mathrm{d}_0 :(\alpha_1, \ldots,\alpha_n)^T\mapsto \sum_{i=1}^n\alpha_i (x_i-1).
\]

Next, define
\[
 \mathrm{d}_1: \bigoplus_{i=1}^m\mathbb{Z}[G] \to \bigoplus_{i=1}^n\mathbb{Z}[G]
\]
by setting
\[
 \mathrm{d}_1: (\beta_r)_{r\in \{r_1,\ldots, r_m\}} \mapsto \sum_{r\in \{r_1,\ldots, r_m\}} (J_{rx}\beta_r)_{x \in X}, \qquad \beta_r \in \mathbb{Z}[G],
\]
where $J_{rx}$ is the image in $\mathbb{Z}[G]$ of the (left) partial derivative $\partial r / \partial x$ (= Fox derivative, see \cite[Sec. 5.15]{MKS}).

Further, the second homotopy $\mathbb{Z}[G]$-module $\pi_2(\mathcal{K}_{G})$ can be viewed as the kernel of the map $\mathrm{d}_1$. This was first observed by K. Reidemeister \cite{R}. More precisely, we have the following result.

\begin{theorem}\label{exact}
There is an exact sequence of $\mathbb{Z}[G]$-modules:
\[
 0 \to \pi_2(\mathcal{K}_G) \xrightarrow{\mathfrak{p}} \bigoplus_{i=1}^m\mathbb{Z}[G] \xrightarrow{\mathrm{d}_1} \bigoplus_{i=1}^n\mathbb{Z}[G] \xrightarrow{\mathrm{d}_0} \mathbb{Z}[G] \xrightarrow{\varepsilon} \mathbb{Z} \to 0,
\]
where $\varepsilon$ is the augmentation map.
\end{theorem}

Take a $\beta \in \pi_2(\mathcal{K}_G)$, we then get $\mathfrak{p}(\beta) = (\beta_1,\ldots,\beta_m)^T$, and hence
\[
  \sum_{j=1}^m\beta_j \dfrac{\partial r_j}{\partial x_i} = 0,
\]
for every $1 \le i \le n$.

Assume now that for a fixed $i$, $\dfrac{\partial r_j}{\partial x_i} = D_jf$, $1 \le j \le m$, then by the previous equality we then have
\begin{equation}\label{foxcommon}
 \sum_{j=1}^m(\beta_j D_j) f = 0,
\end{equation}
it follows that if $\pi_2(\mathcal{K}_G) \ne 0$ then $\mathbb{Z}[G]$ has zero divisors.

\vspace{2ex}

The aim of this paper is thus to construct such groups. We will use Composition--Diamond Lemma technique. More precisely. Consider a free algebra $\mathsf{F}\langle X\rangle$ over a field $\mathsf{F}$. Given polynomials $\varphi, f \in \mathsf{F}\langle X\rangle$. When the polynomial $\varphi$ is divided by $f$? The Composition--Diamond Lemma can help us to answer this question.

\vspace{3ex}

{\bf Acknowledgements:}  the author would like to express his deepest gratitude to \textsc{Dr. Roman Mikhailov} who told about the problem in the course ``Groups and Homotopy Theory'' (YouTube channel ``Lectorium''). Special thanks are due to \textsc{Prof. James Howie} for very useful discussions and for having kindly clarified some very important details. Thanks are due to Czech Technical University in Prague for a great hospitality, where the core part of this paper was written, especially to \textsc{Prof. Pavel \v Stov\'i\v cek} and \textsc{Prof. \v Cestm\'ir Burd\'ik}.

\section{Composition--Diamond Lemma}
Here we present the concepts of Composition--Diamond lemma and Gr\"obner--Shirshov basis. In the classical version of Composition--Diamond lemma, it assumed that considered algebras is over a field, here we consider the general case.

\subsection{CD-Lemma for associative algebras}
Let $\mathsf{K}$ be an arbitrary commutative ring with unit, $\mathsf{K} \langle X \rangle$ the free associative algebra over $\mathsf{K}$ generated by $X$, and let $X^*$ be the free monoid generated by $X$, where empty word is the identity, denoted by $\mathbf{1}_{X^*}$.
Assume that $X^*$ is a well-ordered set. Take $f\in \mathsf{K} \langle X \rangle$ with the leading word (term) $\mathrm{LT}({f})$ and $f= \kappa\mathrm{LT}({f}) + r_{f}$, where $0\neq \kappa\in \mathsf{K}$ and $\mathrm{LT}({r_f})<\mathrm{LT}({f})$. We call $f$ is
\emph{monic} if $\kappa = 1$. We denote by $\mathrm{deg}(f)$ the degree of $\mathrm{LT}({f})$.

A well ordering $\leqslant$ on $X^*$ is called {\it monomial} if for $u,v \in X^*$, we have:
\[
u\leqslant v \Longrightarrow \bigl.w\bigr|_u \leqslant \bigl.w\bigr|_v, \qquad \forall w \in X^*,
\]
where $\bigl.w\bigr|_u : = \bigl.w\bigr|_{x \to u}$ and $x$'s are the same individuality of the letter $x \in X$ in $w$.

A standard example of monomial ordering on $X^*$ is \textit{the deg-lex ordering} (i.e., degree and lexicographical), in which two words are compared first by the degree and then lexicographically, where $X$ is a well-ordering set.

Fix a monomial ordering $\leqslant$ on $X^*$, and let $\varphi$ and $\psi$ be two monic polynomials in $\mathsf{K}\langle X\rangle$. There are two kinds of compositions:
\begin{itemize}
\item[(1)] If $w$ is a word (i.e, it lies in $X^*$) such that $w = \mathrm{LT}({\varphi}) b = a \mathrm{LT}({\psi})$ for some $a,b\in X^*$
with $\mathrm{deg}(\mathrm{LT}({\varphi}))+\mathrm{deg}(\mathrm{LT}({\psi})) >\mathrm{deg}(w)$, then the polynomial $(\varphi,\psi)_w:=\varphi b - a \psi$ is called the {\it intersection composition} of $\varphi$ and $\psi$ with respect to $w$.
\item[(2)] If $w = \mathrm{LT}({\varphi}) = a\mathrm{LT}({\psi}) b$ for some $a,b \in X^*$, then the polynomial $(\varphi,\psi)_w:=\varphi -a\psi b$ is called the {\it inclusion composition} of $\varphi$ and $\psi$ with respect to $w$.
\end{itemize}

We then note that $\mathrm{LT}{(\varphi,\psi)}_w\leq w$ and $(\varphi,\psi)_w$ lies in the ideal $(\varphi,\psi)$ of $\mathsf{K}\langle X\rangle$ generated by $\varphi$ and $\psi$.

Let $\mathbf{S}\subseteq \mathsf{K}\langle X \rangle$ be a monic set (i.e., it is a set of monic polynomials). Take $f \in \mathsf{K}\langle X\rangle$ and $w\in X^*$. We call $f$ is {\it trivial modulo} $(\mathbf{S},w)$, denoted by \[f \equiv 0 \bmod(\mathbf{S},w),\]
if $f = \sum_{s\in \mathbf{S}} \kappa a s b$, where $\kappa \in \mathsf{K}$, $a,b \in X^*$, and $a \mathrm{LT}({s})b \leqslant w$.

A monic set $\mathbf{S}\subseteq \mathsf{K}\langle X \rangle$ is called a \emph{Gr\"obner--Shirshov basis} in $\mathsf{K}\langle X \rangle$ with respect to the monomial ordering $\leq$ if every composition of polynomials in $\mathbf{S}$ is trivial modulo $\mathbf{S}$ and the corresponding $w$.

The following Composition--Diamond lemma was first proved by Shirshov \cite{Sh2} for free Lie algebras over fields (with deg-lex ordering). For commutative algebras, this lemma is known as Buchberger's theorem \cite{Buchberger}.

\begin{theorem}[{\bf Composition Diamond Lemma}]\label{CDA}
Let $\mathsf{K}$ be an arbitrary commutative ring with unit, $\leqslant$ a monomial ordering on $X^*$ and let $I(\mathbf{S})$ be the ideal of $\mathsf{K} \langle X \rangle$ generated by the monic set $\mathbf{S}\subseteq \mathsf{K} \langle X \rangle$. Then the following statements are equivalent:
\begin{itemize}
\item[(1)] $\mathbf{S}$ is a Gr\"obner--Shirshov basis in $\mathsf{K} \langle X \rangle$.
\item[(2)] if $f \in I(\mathbf{S})$ then $\mathrm{LT}({f}) = a\mathrm{LT}({s})b$ for some $s\in \mathbf{S}$ and $a,b\in X^*$.
\item[(3)] the set of irreducible words
   \[\mathrm{Irr}(\mathbf{S}):=\left\{u \in X^*: u \ne a \mathrm{LT}({s})b,\,s \in \mathbf{S},\,a,b\in X^*\right\}\]
is a linear basis of the algebra $\mathsf{K} \langle \bigl.X\bigr|\mathbf{S}\rangle:=\mathsf{K} \langle X\rangle/I(\mathbf{S})$.
\end{itemize}
\end{theorem}

\begin{example}[see \cite{Ufn}]
Let $\mathsf{K}$ be an arbitrary commutative ring and consider the following algebra $\Lambda = \mathsf{K} \langle x,y \rangle/(x^2 - y^2)$. Let us consider the polynomials $\varphi = x^2 - y^2$, $\psi = xy^2 - y^2x$, and let $y \leqslant x$. It is not hard to see that the set $\mathbf{S} = \{\varphi,\psi\}$ is a Gr\"obner--Shirshov basis of $\Lambda$. Indeed,
\begin{eqnarray*}
  (\varphi, \varphi)_{w}&=&{\varphi}x - x {\varphi}\\
   &=& x^3-y^2x - (x^3 - xy^2) = \psi,
\end{eqnarray*}
for $w =x^3,$ and
\begin{eqnarray*}
  (\varphi,\psi)_{w} &=& \varphi y^2 - x\psi \\
  &=&x^2y^2 - y^2y^2 - (x^2y^2 - xy^2x)\\
  &=& \psi x + y^2 \varphi,
\end{eqnarray*}
for $w = x^2y^2.$ Since the set $\mathbf{S}$ is monic, then the set
\[
 \mathrm{Irr}(\mathbf{S}) = \bigcup\limits_{n > 0}\Bigl\{1,x, xy, y^n,y^nx, (xy)^n, (yx)^n, (yxy)^n\Bigr\}
   \]
is the $\mathsf{K}$-basis for $\Lambda$, by Theorem \ref{CDA}.\hfill$\square$
\end{example}

\subsection{CD-Lemma for Semigroups and Groups}

Given a set $X$ consider $S \subseteq X^*\times X^*$ the congruence $\rho(S)$ on $X^*$ generated by $S$, the quotient semigroup
\[
 P = \mathbf{sgr}\langle X \, | \, S \rangle = X^*/\rho(S),
\]
and the semigroup algebra $\mathsf{K}[P]$. Identifying the set $\{u=v\, | \, (u,v) \in S\}$ with $S$, it is easy to see that
\[
 \tau: \mathsf{K}\langle X\, | \, S \rangle \to \mathsf{K}(X^*/ \rho(S)), \quad \sum \kappa f + I(S) \mapsto \sum \kappa \mathrm{LT}(f)
\]
is an algebra isomorphism.

The Shirshov completion $S^c$ of $S$ consists of semigroup relations, $S^c:=\{f-g\}$. Then $\mathrm{Irr}(S^c)$ is a linear $\mathsf{K}$-basis of $\mathsf{K}\langle X \, | \, S \rangle$, and so $\tau(\mathrm{Irr}(S^c))$ is a linear $\mathsf{K}$-basis of $\mathsf{K}(X^*/ \rho(S))$. This shows that $\mathrm{Irr}(S^c)$ consists precisely of the normal forms of the elements of the semigroup $\mathbf{sgr}\langle X \, | \, S \rangle$.

Therefore, in order to find the normal forms of the semigroup $\mathbf{sgr}\langle X \, | \, S \rangle$, it suffices to find a Gr\"obner--Shirshov basis $S^c$ in $\mathsf{K}\langle X \, | \, S \rangle.$ In particular, consider a group $G =  \mathbf{gr}\langle X \, | \, S \rangle$, where $S = \{(u,v) \in F(X) \times F(X)\}$ and $F(X)$ is the free group on a set $X$. Then $G$ has a {\it semigroup presentation}
\[
  G = \mathbf{sgr}\langle X \cup X^{-1} \, | \, S, \, x^{\varepsilon}x^{-\varepsilon} = 1, \varepsilon = \pm, x \in X \rangle, \qquad X \cap X^{-1} = \varnothing,
\]
as a semigroup.

\section{First Examples of Groups}
Let $G = \mathbf{gr} \langle x, y_1,\ldots, y_\ell\, | \, r_{11}=r_{12}, \ldots, r_{n1} = r_{n2} \rangle$ be a group and let $\dfrac{\partial r_i}{\partial x} = D_i f$ for $1 \le i \le n$, here $r_i:=r_{i1}r^{-1}_{i2}$, and $D_i, f\in \mathbb{Z}[G]$. We assume that every $r_{ij}$ does not contain other term $r_{pq}$ as a subword, and all $r_i$ are not reduced words, \textit{i.e.,} they do not contain a word (as a subword) of form $aa^{-1}$.

Consider now $G$ as a semigroup and set $x >x^{-1} > y_{j}$ for $1 \le j \le \ell$, and deg-lex order the free monoid $\mathfrak{W}$ generated by ${x,x^{-1},y_1,y_1^{-1},\ldots, y_\ell, y_\ell^{-1}}$.

Fix $1 \le i \le n$ and consider $r_i$, we have
\[
 \dfrac{\partial r_{i}}{\partial x} = \dfrac{\partial r_{i1}}{\partial x} - \dfrac{\partial r_{i2}}{\partial x}.
\]

Without loss of generality, we may put $r_{i1}>r_{i2}$ for $1 \le i \le n$. Hence $\mathrm{LT}( \partial r_{i}/ \partial x) = \mathrm{LT}(\partial r_{i1}/\partial x)$.

Set $\mathrm{LT}( \partial r_{i} /\partial x ) = u_i\bar f$, then $r_{i1} = u_i\bar f x \widetilde{u_i}$, where $\widetilde{u_i}$ does not involve $x$ and $x^{-1}$. We have
\[
 \varphi_0(r_i):=\dfrac{\partial r_{i}}{\partial x} = \dfrac{\partial u_{i}}{\partial x} + u_i\dfrac{\partial \bar f}{\partial x} - \dfrac{\partial r_{i2}}{\partial x}  + u_i\bar f.
\]

Then
\begin{eqnarray*}
  \varphi_1(r_i) &:=& \left( \dfrac{\partial r_{i}}{\partial x},f \right)_{u_i\bar f} := \dfrac{\partial r_{i}}{\partial x} - u_if \\
  &=& \dfrac{\partial u_{i}}{\partial x} + u_i\dfrac{\partial \bar f}{\partial x}- \dfrac{\partial r_{i2}}{\partial x}  + u_i\bar f - u_if \\
  &=&\dfrac{\partial u_{i}}{\partial x} + u_i\dfrac{\partial \bar f}{\partial x} - \dfrac{\partial r_{i2}}{\partial x} - u_i f_1,
\end{eqnarray*}
where $f_1:=f - \bar f$.

We thus have to consider the following possibilities: (1) $\varphi_1(r_i) = 0$, (2) $\mathrm{LT}(\varphi_1(r_i)) = u_i\mathrm{LT}(\partial \bar f / \partial x),$ (3) $\mathrm{LT}(\varphi_1(r_i)) = \mathrm{LT}(\partial r_{i2}/\partial x)$, (4) $\mathrm{LT}(\varphi_1(r_i)) = u_i\mathrm{LT}(f_1)$.

To get some examples we consider cases (1) and (3). Other cases will be considered the next section.

(1) Let $\varphi_1(r_i) = 0$. We have
\[
 \dfrac{\partial u_{i}}{\partial x} + u_i\dfrac{\partial \bar f}{\partial x} - \dfrac{\partial r_{i2}}{\partial x} - u_i f_1 = 0.
\]

If we assume that the $\partial u_i / \partial x, \partial r_{i2}/\partial x \ne 0$ have common terms it then implies that a $\mathstrut^\bullet r_{i2}$ contain $\mathstrut^\bullet u_i$ as a subword, here, for a word $w\in \mathfrak{W}$, we set $w = \mathstrut^\bullet w w^\bullet$. Similarly, one can easy see that the polynomials $u_i  \partial \bar f/\partial x$ $ \partial r_{i2}/\partial x$ have no similar terms.

Thus we may put $\partial u_i/\partial x = 0$, $ \partial \bar f/\partial x = f_1$ and $ \partial r_{i2}/\partial x = 0$, {\it i.e.,} $r_{i2}$ does not involve the terms $x$, $x^{-1}$.

(3) Let $\mathrm{LT}(\varphi_1(r_i)) = \mathrm{LT}\left( \dfrac{\partial r_{i2}}{\partial x} \right) = v_i\bar f$. Hence $r_{i2} = v_i\bar f x \widetilde{v_i}$, where $\widetilde{v_i}$ does not involve $x$ and $x^{-1}$, and
\[
 \dfrac{\partial r_{i2}}{\partial x} = \dfrac{\partial v_{i}}{\partial x} + v_i\dfrac{\partial \bar f}{\partial x} + v_i \bar f.
\]

We then get
\begin{eqnarray*}
  \varphi_2(r_i) &:=& (-\varphi_1, f)_{v_i\bar f} := -\varphi_1 - v_if \\
  &=& -\dfrac{\partial u_{i}}{\partial x} - u_i\dfrac{\partial \bar f}{\partial x} +  \dfrac{\partial v_{i}}{\partial x} + v_i\dfrac{\partial \bar f}{\partial x} + v_i \bar f + u_i f_1 - v_if \\
  &=& \dfrac{\partial v_{i}}{\partial x} - \dfrac{\partial u_{i}}{\partial x} + (v_i - u_i)\dfrac{\partial \bar f}{\partial x} + (u_i-v_i)f_1 \\
  &=& \dfrac{\partial v_{i}}{\partial x} - \dfrac{\partial u_{i}}{\partial x} + (v_i-u_i)\left(\dfrac{\partial \bar f}{\partial x} - f_1 \right).
\end{eqnarray*}

Setting $f_1 =\partial \bar f / \partial x$ we then get $\varphi_2(r_i) = \partial v_i / \partial x - \partial u_i / \partial x$. Thus we have the same problem as for $\varphi_0(r_i):=\partial r_{i1}/ \partial x - \partial r_{i_2} / \partial x$. It follows that we then get the following set of groups
\begin{equation}\label{G1}
 G_{\ell,n} = \mathbf{gr}\langle x, y_1,\ldots, y_\ell\, | \, r_{11} = r_{12}, \ldots, r_{n,1} = r_{n,2}\rangle,
\end{equation}
where\begin{eqnarray*}
  r_{i1} &=& \mathsf{u}_{i,1}\mathrm{w}x\mathsf{u}_{i,2} \cdots \mathsf{u}_{i,p-1}\mathrm{w}x\mathsf{u}_{i,p_i}, \\
  r_{i2} &=&  \mathsf{v}_{i,1}\mathrm{w}x\mathsf{v}_{i,2} \cdots \mathsf{v}_{i,q-1}\mathrm{w}x\mathsf{v}_{i,q_i},
\end{eqnarray*}
here for $1 \le i \le n$, $p_i \ge 1$, $q_i \ge 0$, and if $q_i=0$ then $r_{i2} = \mathsf{v}_{i,0}$, further all $r_{i1}$, $r_{i2}$ are not reduced, all $\mathsf{u}_{i,j},\mathsf{v}_{i,k}$ do not involve $x,x^{-1}$, $\mathrm{w}\ne 1$, and every term of any relation does not contain, as a subword, a term of other relations. Therefore we get
\begin{theorem}\label{themainresult}
  For a group $G_{\ell,n}$ presented by (\ref{G1}), with $\pi_2(\mathcal{K}_{G_{\ell,n}}) \ne 0$, the group ring $\mathbb{Z}[G_{\ell,n}]$ has nontrivial zero divisors.
\end{theorem}
\begin{proof}
  Indeed, for all $1 \le i \le n$, we have
  \begin{eqnarray*}
    \dfrac{\partial r_i}{\partial x} &=& \mathsf{u}_{i,1} \left(\dfrac{\partial \mathrm{w}}{\partial x} + \mathrm{w}\right) + \cdots + \mathsf{u}_{i,1}\mathrm{w}x\cdots \mathsf{u}_{i,p-1} \left( \dfrac{\partial \mathrm{w}}{\partial x} + \mathrm{w} \right) \\
    && - \mathsf{v}_{i,1} \left(\dfrac{\partial \mathrm{w}}{\partial x} + \mathrm{w}\right) - \cdots - \mathsf{v}_{i,1}\mathrm{w}x\cdots \mathsf{v}_{i,q-1} \left( \dfrac{\partial \mathrm{w}}{\partial x} + \mathrm{w} \right),
  \end{eqnarray*}
  hence
  \[
    \dfrac{\partial r_i}{\partial x} = \left(\sum_{k=p-1}^{1} \mathsf{u}_{i,1}\mathrm{w}x \cdots \mathsf{u}_{i,p-k} - \sum_{t=q-1}^{1} \mathsf{v}_{i,1}\mathrm{w}x \cdots \mathsf{v}_{i,q-t}  \right) \left( \dfrac{\partial \mathrm{w}}{\partial x} + \mathrm{w} \right).
  \]

If $\pi_2(\mathcal{K}_{G_{\ell,n}}) \ne 0$, then by (\ref{exact}), we obtain nontrivial zero divisors in $\mathbb{Z}[G]$, as claimed.
\end{proof}

\section{The Other Possibilities}

In this section we consider other possibilities which appeared in the construction of $G_{\ell,n}$ and we will see that we again get the same set (\ref{G1}) of groups.

\begin{lemma}\label{commondiv}
  Let $\mathfrak{w}, \mathfrak{p}_1, \ldots, \mathfrak{p}_\ell \in \mathfrak{F}$ and $P = \sum_{i=1}^\ell \varepsilon_i \mathfrak{p}_i \in \mathbb{Z}[\mathfrak{F}]$, where $\varepsilon = \pm 1$. If the polynomials $\partial \mathfrak{w} / \partial x$, $P$ have a common term, say $\mathfrak{p}_k$, then the words $\m{w}, \m{p}_k$ have common left divisor, i.e., there exist nonempty words $\m{u}, \m{w}',\m{p}_k' \in \m{F}$ such that $\m{w} = \m{u}\m{w}'$, $\m{p}_k = \m{u}\m{p}_k'$.
\end{lemma}
\begin{proof}
  Indeed, let $\m{w} = \m{w}_1 x^{n_1} \m{w}_2 x^{n_2} \cdots \m{w}_m x^{n_m}\m{w}_{n_{m+1}}$, where for $1 \le j \le n_{m+1}$ every $\m{w}_{j}$ does not involve $x$, $x^{-1}$ and $n_j \in \mathbb{Z}$. Thus we have
  \[
   \dfrac{\partial \m{w}}{\partial x} = \m{w}_1 \dfrac{\partial x^{n_1}}{\partial x} + \m{w}_1 x^{n_1} \m{w}_2 \dfrac{\partial x^{n_2}}{\partial x} + \cdots + \m{w}_1 x^{n_1} \m{w}_2 x^{n_2} \cdots \m{w}_m \dfrac{\partial x^{n_m}}{\partial x},
  \]
  where
  \[
  \dfrac{\partial x^{n_j}}{\partial x} = \begin{cases}
    1 + x + \cdots + x^{n_j-1}, & n_j \ge 0, \\
    -x^{-1} - x^{-2} - \cdots - x^{-|n_j|}, & n_j < 0,
  \end{cases}
  \]
  and the statement follows.
\end{proof}

Give a group $G = \mathbf{gr} \langle x,  y_1,\ldots, y_\ell \, | \, r_1,\ldots, r_m \rangle $. Take a $\bar f \in \mathfrak{F}$. Let $r_1 = \{r_{11} = r_{12}\}$ and let $\mathrm{LT}(\partial r_{11}/\partial x) = u_1 \bar f$. Hence, $r_{11} = u_1\bar f x\widetilde{u_1}$ with $\widetilde{u_1} \ne \widetilde{u_1}(x)$.

We get
\[
 \dfrac{\partial r_1}{\partial x} = \dfrac{\partial r_{11}}{\partial x} - \dfrac{\partial r_{12}}{\partial x} = \dfrac{\partial u_1}{\partial x} + u_1 \dfrac{\partial \bar f}{\partial x} - \dfrac{\partial r_{12}}{\partial x}.
\]

Then
\[
 \varphi_1: = \left( \dfrac{\partial r_1}{\partial x}, f \right)_{u_1\bar f} = \dfrac{\partial r_1}{\partial x} - u_1f = \dfrac{\partial u_1}{\partial x} + u_1 \dfrac{\partial \bar f}{\partial x} - \dfrac{\partial r_{12}}{\partial x} - u_1f_1.
\]
where $f_1:=f-\bar f.$

\subsection{A monomial $\bar f$ involves $x^{\pm 1}$}
Let $\bar f = w_1 x^{n_1} w_2 \cdots w_k x^{n_k} w_{k+1}$, where $n_i \ne 0$, $w_i \ne w_{i}(x)$ for $i=1,\ldots, k+1$.

\begin{lemma}\label{LT(udf)}
  Let $\mathrm{LT}(\varphi_1) = \mathrm{LT} (u_1 \partial \bar f/\partial x) = u_2 \bar f$. Then $\bar f = wxw\cdots wxw$ and $u_1 = u_2wx$.
\end{lemma}
\begin{proof}
   Let $\bar f = w_1 x^{n_1} w_2 \cdots w_k x^{n_k} w_{k+1}$, where $n_i \ne 0$, $w_i \ne w_{i}(x)$ for $i=1,\ldots, k+1$.

   (1) Assume that $n_k >0$ then $\dfrac{\partial x^{n_k}}{\partial x} = 1 +x + \cdots + x^{n_k-1}$, and we get

   \[
    \dfrac{\partial (f_\lambda \cdot x^{n_k} w_k)}{\partial x} = \dfrac{\partial f_\lambda}{\partial x} + f_\lambda(1+x \cdots + x^{n_k-1}),
   \]
   where $f_\lambda := w_1x^{n_1}\cdots w_{k}$.

   Then, by assumption,
   \[
    \mathrm{LT}\left(u_1 \dfrac{\partial \bar f}{\partial x}\right) = u_1 w_1x^{n_1}\cdots w_{k} x^{n_k-1} = u_2 w_1 x^{n_1} w_2 \cdots w_k x^{n_k} w_{k+1},
   \]
   hence either $w_{k+1}=1$ or $n_{k}=1$.

   Let $w_{k+1} = 1$, $n_k>1$, then $u_1 w_1x^{n_1}\cdots w_{k} = u_2 w_1 x^{n_1} w_2 \cdots w_k x$, and hence $w_k = 1$. Thus, $u_1w_1x^{n_1} \cdots w_{k-1}x^{n_{k-1}} = u_2w_1x^{n_1}w_2 \cdots w_{k-1}x^{n_{k-1}+1}$, hence we must put $w_{k-1} = \cdots = w_1 = 1$ and we then obtain $u_1x^{n_1 + \cdots n_{k-1}} = u_2 x^{n_1 + \cdots n_{k-1} +1}$. Hence $u_1 = u_2x$.

   Let $w_{k+1} \ne 1$, $n_k=1$, we then have
    \[
    u_1 w_1x^{n_1}\cdots xw_{k} = u_2 w_1 x^{n_1} w_2 \cdots w_k x w_{k+1}.
    \]

    Hence, $w_{k+1} = w_k = \cdots w_1= w$ and $u_1 = u_2wx$.

    (2) Assume that $n_k <0$ then $\dfrac{\partial x^{n_k}}{\partial x} = - (x^{-1} + x^{-2} + \cdots + x^{n_k})$, hence $\mathrm{LT} \left( u_1 \dfrac{\partial \bar f}{\partial x} \right) = - u_1 w_1 x^{n_1} \cdots w_k x^{n_k}$. By $\mathrm{LT}\left(u_1 \dfrac{\partial \bar f}{\partial x}\right) = u_2 \bar f$, $w_{k+1} = 1$, and $u_1 = -u_2$. But it follows that $r_{11}$ is not a reduced word.
\end{proof}

\begin{proposition}
  If $\mathrm{LT}(\varphi_1) = \mathrm{LT}(u_1 \partial \bar f / \partial x) = u_2\bar f$ and $\varphi_k(r_1)=0$ for $k \ge 2$ we then get the set of group are described by (\ref{G1}).
  \end{proposition}
\begin{proof}
We have 
\[
 \varphi_2(r_1):=\varphi_1(r_1) - u_2f = \dfrac{\partial u_1} {\partial x} + u_1 \dfrac{\partial \bar f}{ \partial x} - \dfrac{\partial r_{12} }{ \partial x} - u_1 f_1 - u_2 f.
\]

By Lemma \ref{LT(udf)}, we have $\bar f = (wx)^nw$, $u_1 = u_2wx$.

We get
\[
 \dfrac{\partial \bar f}{\partial x} = w + wxw + \cdots + (wx)^{n-1}w, \qquad \dfrac{\partial u_1}{\partial x} = \dfrac{\partial u_2}{\partial x} + u_2w.
\]

Then

\begin{eqnarray*}
  \varphi_2(r_1) &=& \dfrac{\partial u_2}{\partial x} + u_2w + u_2wx(w + wxw + \cdots + (wx)^{n-1}w) \\
  &&- \dfrac{\partial r_{12}}{\partial x} - u_2wxf_1 -u_2\bar f - u_2 f_1 \\
  &=& \dfrac{\partial u_2}{\partial x} + u_2w + u_2wx(w + wxw + \cdots + (wx)^{n-2}w) \\
  &&- \dfrac{\partial r_{12}}{\partial x} - u_2wxf_1 - u_2 f_1 \\
  &=& \dfrac{\partial u_2}{\partial x} + u_2 (w + wxw + \cdots + (wx)^{n-1}w) - \dfrac{\partial r_{12}}{\partial x} - (u_2wx + u_2)f_1\\
  &=& \dfrac{\partial u_2}{\partial x} + u_2 \dfrac{\partial \bar f}{\partial x} - \dfrac{\partial r_{12}}{\partial x} - (u_2wx + u_2)f_1.
\end{eqnarray*}

Suppose that $\varphi_2(r_1) = 0$. By Lemma \ref{commondiv}, $\partial u_2 / \partial x, \partial r_{12}/ \partial x = 0$. Thus we have the following equation
\[
 w + wxw + \cdots + (wx)^{n-1}w - wx f_1 - f_1 = 0,
\]
which has a solution. Indeed, we may put $n=2$ and $f_1 = w$.

It follows that we get the following set of groups $G = \langle x, y_1,\ldots,y_\ell \, |\, r_{11} = r_{12}, \ldots, r_{1m} = r_{2m} \rangle$, where $r_{i1} = u_i wxwxw u_i'$, $r_{i2} = v_i$, $1\le i \le m$, here the $w$ and all the words $u_i,, u_i',v_i$ do not involve $x$, $x^{-1}$, and they each of them does not contain as a subword another word. Thus, we have got the set of groups (\ref{G1}).

Assume now that $\varphi_2(r_1) \ne 0$ and $\mathrm{LT}(\varphi_2(r_1)) = \mathrm{LT}(\partial r_{12}/ \partial x) = v_1\bar f$. Then $r_{12} = v_1\bar f x v_1'$, where $v_1'$ does not involve $x$, $x^{-1}$.

We have
\begin{eqnarray*}
 \varphi_3(r_1) &: =& -\varphi_2(r_1)-v_1f \\
 &=& -\dfrac{\partial u_2}{\partial x} - u_2 \dfrac{\partial \bar f}{\partial x} + (u_2wx + u_2)f_1 \\
 && \dfrac{\partial v_1}{\partial x} + v_1 \dfrac{\partial \bar f}{\partial x} + v_1\bar f - v_1f \\
 &=& \dfrac{\partial v_1}{\partial x} -\dfrac{\partial u_2}{\partial x} - u_2 \dfrac{\partial \bar f}{\partial x} + (u_2wx + u_2)f_1 + v_1 \dfrac{\partial \bar f}{\partial x} - v_1f_1.
\end{eqnarray*}

If we put $\varphi_3(r_1) = 0$ then by Lemma \ref{commondiv}, $\partial v_1 / \partial x, \partial u_2 / \partial x = 0$ because of $u_2$, $v_1$ have no common left divisors by assumption. Thus, to have a common divisor for $\partial r_i / \partial x$ we have to put $\varphi_3(r_1) \ne 0$ and $\mathrm{LT}(\varphi_3(r_1)) = \mathrm{LT}(v_1 \partial \bar f / \partial x) = v_2 \bar f$. Similarly, as for $\varphi_2(r_1)$, we then obtain
\[
  \varphi_4(r_1) = \dfrac{\partial v_2}{\partial x} -\dfrac{\partial u_2}{\partial x} - u_2 \dfrac{\partial \bar f}{\partial x} + (u_2wx + u_2)f_1 + v_2 \dfrac{\partial \bar f}{\partial x} - (v_2wx + v_2)f_1.
\]

Hence the equation has a solution. Indeed, we may put $\partial v_2/ \partial x, \partial u_2 / \partial x = 0$, $\bar f = (wx)^2w$, $f_1 = w$. Thus we have the same groups are described by (\ref{G1}).
\end{proof}

\subsection{A monomial $\bar f$ does not involve $x^{\pm 1}$}

We then have $\partial r_i / \partial x = \partial u_i / \partial x - \partial r_{i2} / \partial x$, and if we put $\mathrm{LT}(\partial r_i / \partial x) = \mathrm{LT}(\partial u_{i1} / \partial x) = u_{i2}\bar f$. It follows that $u_{i1} = u_{i2}\bar f x u_{i2}'$ where $u_{i2}'$ does not involve $x$, $x^{-1}$. It is easy to see that it is impossible.

Indeed,we have
\begin{eqnarray*}
  \varphi_1(r_i) &=& \dfrac{\partial r_{i}}{\partial x} - u_{i2} f \\
  &=& \dfrac{\partial u_{i2}}{\partial x} + u_{i2}\bar f - \dfrac{\partial r_{i2}}{\partial x} - u_{i2}f \\
  &=& \dfrac{\partial u_{i2}}{\partial x}  - \dfrac{\partial r_{i2}}{\partial x} - u_{i2}f_1,
\end{eqnarray*}
hence by Lemma \ref{commondiv}, $\varphi_k(r_i) \ne 0$ for $k\ge 1$.

\section*{Conclusion}
We have seen that homological point of view on the Kaplansky's zero divisor conjecture is very useful and very easy for understanding. The author is going to study when this groups have no a torsion and when $\pi_2(\mathcal{K}_{G_{\ell,n}}) \ne 0$ in the future papers.

\end{document}